\documentclass[preprint,pra,superscriptaddress]{revtex4}

\usepackage[dvips]{graphicx} % for figures
\usepackage{amsfonts}
\usepackage{amssymb}
\usepackage{amscd}
\usepackage{amsmath}    % need for subequations
\usepackage{amsthm}   % for definition etc.
\usepackage{enumerate}
\usepackage{epsfig}
\usepackage{subfigure}

\newtheorem{theorem}{Theorem}[section]
\newtheorem{lemma}[theorem]{Lemma}

\newtheorem{definition}[theorem]{Definition}

\begin{document}

% title
\title{A family of Quadratic Resident Codes over $Z_{2^{m}}$}
% author

\author{Xiongqing Tan}
\email{ttanxq@jnu.edu.cn}
%\author{Norbert L\"utkenhaus}
%\email{nlutkenhaus@iqc.ca}
\affiliation{%
Department of Mathematics, College of Information Science and Technology,\\
Jinan University, Guangdong, P. R. China \\
}%
\affiliation{%
Department of Physics and Department of Electrical \& Computer Engineering, \\
University of Toronto, Toronto, Ontario, Canada \\
}%

%%%%%%%%%%%%%%%%%%%%%%%%%%%%%%%%%%%%%%%%%%%%%%%%%%%%%%%%%%%%%%%%%%%%%%%
% Abstract
%%%%%%%%%%%%%%%%%%%%%%%%%%%%%%%%%%%%%%%%%%%%%%%%%%%%%%%%%%%%%%%%%%%%%%%
\begin{abstract}

A cyclic codes of length $n$ over the rings $Z_{2^{m}}$ of integer of modulo $2^{m}$ is a linear code with property that if the codeword $(c_0,c_1,...,c_{n-1})\in \mathcal{C}$ then the cyclic shift $(c_1,c_2,...,c_0)\in \mathcal{C}$. Quadratic residue codes are a particularly interesting family of cyclic codes. We define such family of codes in terms of their
idempotent generators and show that these codes
also have many good properties which are analogous in many respects to properties of binary quadratic
residue codes. Such codes constructed are self-orthogonal. And we also discuss their hamming weight.\\

\textbf{Keywords.} Generators and idempotents of $Z_{2^{m}}$-cyclic codes, $Z_{2^{m}}$ quadratic residue codes, self-orthogonal codes.\\

\textbf{2000 MR Subject Classification.} 94B05

%R. Raz, O. Reingold, and S. Vadhan [Extracting all the randomness and reducing the error in Trevisan's extractors, Proc 31st ACM Symp Theory of Computing, Atlanta, GA, May 1999, to appear in J Comput System Sci Special Issue on STOC 99]

\end{abstract}

\maketitle

% main text
\section{Introduction}
Hammons \textit{et al.}[1] have shown that some famous nonlinear binary codes, such as Kerdock, Preparata,
Goethals, Delsarte-Goethals codes, are all images under the Gray map of
some linear codes over $Z_{4}$, that is the ring of integers modulo 4. Linear codes over rings have drawn attention of many
researchers in past decade. The binary quadratic residue codes have a
significant place in the literature of algebraic coding theory, in
particular as important examples of cyclic codes, because they can be defined well by their idempotent generators and have
many good error correction properties.

Quadratic residue codes over $Z_{4}$ were introduced by Bonnecaze \textit{et al.}[2] and studied by Pless and
Qian [3]. Chui \textit{et al.} discussed the cyclic codes and
quadratic residue codes over $Z_{8}$ [4]. Calderbank and
Sloane thought about the problem what codes are the obtained if one lift binary generator polynomial to $Z_{8},Z_{16},...$ and even to the 2-adic integers $Z_{2^{\infty}}$ [5]. Kanwor generalized these results to quadratic residue codes over $Z_{q^{m}}$ of length $p$,
where $p\equiv \pm 1(\text{mod} \quad 4q )$ is an odd prime and $q$ is
another prime to be as a quadratic residue modulo $p$ [6]. Abulrub and Oehmke have gotten some results on the generators of ideals in the ring $Z_{4}[x]/(x^{n}-1)$ and discussed about the $Z_{4}$ cyclic codes of length $2^{e}$ [7]. Wolfman considered what kinds of codes $C$ under the Gray map $\phi$ over $Z_{4}$ such that $\phi(C)$ is linear? Cyclical? Linear cyclical? [8,9] Another researchers also got some specific representation of quadratic residue codes over $Z_{16}$ or $Z_{32}$ [10,11].

Let
\begin{equation} \label{QR:idempotentbinary}
e_{1}=\sum_{i\in Q}x^{i}  \quad \text{and} \quad e_{2}=\sum_{i\in N}x^{i},
\end{equation}
where $Q$ is the set of quadratic residues and $N$ is the set of nonresidues for $p\equiv \pm 1(\text{mod} \quad 8)$. As we known, $e_{1}$ and $e_{2}$ are idempotents of binary
$[p,(p+1)/2]$ QR codes when $p\equiv
-1(\text{mod} \quad 8)$; $e_{1}$ and
$e_{2}$ are idempotents of binary $[p,(p-1)/2]$ QR codes when $p\equiv 1(\text{mod} \quad 8)$. We can
define quadratic residue codes of length $p$ over $Z_{2^{m}}$, that is the ring of integers modulo
$2^{m}$ and $m\geq 2$ is any positive integer, in terms
of their idempotent generators
\begin{equation} \label{QR:idempotentZ2m}
\alpha+\beta e_{1}+\gamma e_{2}\qquad (\alpha,\beta,\gamma \in
Z_{2^{m}}).
\end{equation}
In this correspondence we show that these codes also have many good
properties similar to binary quadratic residue codes.

\section{Preliminaries}
We use the symbol $Z_{2^{m}}$ to denote the ring of integers modulo $2^{m}$. A linear code $\mathcal{C}$ of length $n$ over $Z_{2^{m}}$ is defined to be an additive submodule of the
$Z_{2^{m}}$-module $Z_{2^{m}}^n$. A cyclic codes of length $n$ over the rings $Z_{2^{m}}$ is a linear code with property that if the codeword $(c_0,c_1,...,c_{n-1})\in \mathcal{C}$ then the cyclic shift $(c_1,c_2,...,c_0)\in \mathcal{C}$. We let
$$
R_{n}=Z_{2^{m}}[x]/(x^{n}-1).
$$
As is customary on the polynomial ring $R_{n}$, the elements are identified with polynomials of
degree less than or equal to $n-1$. A $n$-tuple
$(a_{0},a_{1},a_{2},\cdots,a_{n-1})$ in $Z_{2^{m}}$ can be
identified with the element $a_{0}+a_{1}x+\cdots+a_{n-1}x^{n-1}$
over $R_{n}$. Then the polynomial representation of a cyclic code $\mathcal{C}$ is an ideal of $R_{n}$. An idempotent in $Z_{2^{m}}[x]$ is defined to be a polynomial
$e(x)$ such that $e(x)^{2}\equiv e(x)(\text{mod} \quad
x^{n}-1)$. We define the dual of a cyclic code $\mathcal{C}$ to be the set
$$
\mathcal{C}^{\perp}=\{u\in Z_{2^{m}}^n: \quad u\cdot v=0 \quad \text{for all} \quad v \quad \text{in} \quad \mathcal{C}\},
$$
where the denotation $\cdot$ is an inner product over $Z_{2^{m}}$. It is clear that $\mathcal{C}^{\perp}$ is also a cyclic code.The notation of self-orthogonal codes ($\mathcal{C}\subseteq \mathcal{C}^{\perp}$) and self-dual codes ($\mathcal{C}=\mathcal{C}^{\perp}$) are defined in standard way [12,13,14]. If a $Z_{2^{m}}$-cyclic code $\mathcal{C}$ has the idempotent generator $e(x)$, then $\mathcal{C}^{\perp}$ has the
idempotent generator $1-e(x^{-1})$. If
$\mathcal{C}_{1}$ and $\mathcal{C}_{2}$ are $Z_{2^{m}}$-cyclic codes with idempotent
generators $e_{1}(x)$ and $e_{2}(x)$ respectively, then
$\mathcal{C}_{1}\bigcap \mathcal{C}_{2}$ has the idempotent generator $e_{1}(x)e_{2}(x)$ and $\mathcal{C}_{1}+\mathcal{C}_{2}$ has the idempotent generator $e_{1}(x)+e_{2}(x)-e_{1}(x)e_{2}(x)$.

Let $p\equiv \pm 1(\text{mod} \quad 8)$ be a odd prime. This assumption is always existing for $p$ in this paper. Suppose that $\xi$ is a primitive $p$-th root of
unity in the extension field of the finite field $GF(2)$. Let
$$
f_{Q_{2}}(x)=\prod\limits_{i\in Q}(x-\xi^{i})\quad \text{and} \quad
f_{N_{2}}(x)=\prod\limits_{i\in N}(x-\xi^{i}).
$$
The degree of $f_{Q_{2}}(x)$ and $f_{N_{2}}(x)$ both are
$\frac{p-1}{2}$, and
$$
x^{p}-1=(x-1)f_{Q_{2}}(x)f_{N_{2}}(x).
$$
The binary quadratic residue codes
$Q_{2},N_{2},Q_{2}^{'},N_{2}^{'}$ are the cyclic codes generated
by $f_{Q_{2}}(x),f_{N_{2}}(x),(x-1)f_{Q_{2}}(x),(x-1)f_{N_{2}}(x)$,
respectively.

From [6], we conclude that quadratic residue codes $Q_{2^{m}},N_{2^{m}},Q_{2^{m}}^{'},N_{2^{m}}^{'}$ with the length $p$ can generated by $f_{Q_{2^{m}}}(x),f_{N_{2^{m}}}(x),(x-1)f_{Q_{2^{m}}}(x),(x-1)f_{N_{2^{m}}}(x)$, where $f_{Q_{2^{m}}}(x)$ and $f_{N_{2^{m}}}(x)$ are the Hensel lifts of $f_{Q_{2}}(x)$ and $f_{N_{2}}(x)$.
We call $Q_{2^{m}},N_{2^{m}},Q_{2^{m}}^{'},N_{2^{m}}^{'}$ as $Z_{2^{m}}$-quadratic residue codes or $Z_{2^{m}}$-QR codes.

\  \ Quadratic residue codes are cyclic codes which can be defined in terms of their idempotent generator. Theorem
2.0.1 of [6] have showed that the existence and uniqueness of idempotent generator of $Z_{2^{m}}$ cyclic code.

\begin{lemma}(\text{[6]} Theorem 2.0.1)\label{Z2mQR:Lemma:IdempotentQR}
Let $C$ be a $Z_{2^{m}}$ cyclic code of odd length $n$. If
$C=(f)$, where $fg=x^{n}-1$ for some $g$ such that $f$ and $g$ are
coprime, then $C$ has an idempotent generator in
$R_n$. Moreover, the idempotent generator of
the cyclic code is unique.
\end{lemma}

 At the same time, we have known that the numbers of quadratic residues or non-residues of an odd prime $p$ [6]. 0 is a quadratic residue of $p$ in [6], but in fact, a quadratic residue modulo $p$ is a number that should be coprime with $p$ firstly. For our purpose, we consider the numbers of 0 in the set $\{i+j\}$ when $i,j$ are quadratic residues or non-residues.

\begin{lemma}\label{Z2mQR:Lemma:iaddjp-1}
Let $p\equiv -1(\text{mod} \quad 8)$. \\
(1) For $\forall i,j\in Q$, $i+j\neq 0$; For $\forall i,j\in N$, $i+j\neq 0$;\\
(2) For $\forall i\in Q$, $\exists j\in N$, such that $i+j=0$.
\end{lemma}

\begin{proof}(1) Suppose that for some $i\in Q$,
$\exists j\in Q$ such that $i+j=0\Rightarrow -j=i\in Q$. Since
$p\equiv -1(\text{mod} \quad 8)\Rightarrow -1\in N$. So if $j\in
Q\Rightarrow -j\in N$. Then we get a contradiction. So for
$\forall i,j\in Q, i+j\neq 0$. Similar proof for $\forall i,j\in N$, $i+j\neq 0$. \\
(2) Since $-1\in N$, for $\forall i\in Q$, we let $j=-i\in N$. Then we have found such
$j$ satisfy to $i+j=0$ and $j\in N$.
\end{proof}

Similarly, we can get  the following lemma when $p\equiv 1(\text{mod} \quad 8)$. In this case, $-1$ is a quadratic
residue modulo $p$.

\begin{lemma}\label{Z2mQR:Lemma:iaddjp1}
Let $p\equiv 1(\text{mod} \quad 8)$. \\
(1) For $\forall i\in Q$, $\exists j\in Q$, such that $i+j=0$;
For $\forall i\in N$, $\exists j\in N$ such that $i+j=0$; \\
(2) For $\forall i,j\in N$, $i+j\neq 0$.
\end{lemma}

 From these previous lemmas, we can get some properties of
$e_{i}e_{j}(i=1 \mbox{ or }2)$ as follows.
\begin{lemma}\label{Z2mQR:Lemma:e1^2e2^2e1*e2p-1}
If $p\equiv -1(\text{mod} \quad 8)$, namely $p=8k-1$, then in
$R_{p}$: $e_{1}^{2}=(2k-1)e_{1}+2ke_{2}$,
$e_{2}^{2}=(2k-1)e_{2}+2ke_{1}$,
$e_{1}e_{2}=(2k-1)(1+e_{1}+e_{2})+2k$.
\end{lemma}
\begin{proof} Since $e_{1}=\sum_{i\in
Q}x^{i},e_{2}=\sum_{i\in N}x^{i}$, then
$$ e_{1}^{2}=\sum_{i,j\in Q}x^{i+j},\quad e_{2}^{2}=\sum_{i,j\in
N}x^{i+j},\quad e_{1}e_{2}=\sum_{i\in Q,j\in N}x^{i+j}.$$

For $e_{1}^{2}=\sum_{i,j\in Q}x^{i+j}$, we obviously know that the set $\{i+j|j\in
Q\}$ has $2k-1$ quadratic residue, $2k$ non-residues for every $i\in Q$, and no 0,
by Theorem 8.3(1) [6] and Lemma \ref{Z2mQR:Lemma:iaddjp-1}. Thus each
quadratic residue appears $2k-1$ times and each non-residue appears $2k$ times in the set $\{i+j|i,j\in Q\}$. Hence
$e_{1}^{2}=(2k-1)e_{1}+2ke_{2}$.

Similar to $e_{2}^{2}=\sum_{i,j\in N}x^{i+j}$, we obtain $e_{2}^{2}=(2k-1)e_{2}+2ke_{1}$  by Theorem 8.3(2) [6] and Lemma \ref{Z2mQR:Lemma:iaddjp-1}.

For $e_{1}e_{2}=\sum_{i\in Q,j\in N}x^{i+j}$, for every $i\in Q$, the
set $\{i+j|j\in N\}$ has $2k-1$ quadratic residues, $2k-1$
non-residues, and one 0, by Theorem 8.3(2) [6] and Lemma \ref{Z2mQR:Lemma:iaddjp-1}. Hence, in the set $\{i+j|i\in
Q, j\in N\}$, each quadratic residue and non-residue
both appear $2k-1$ times, but 0 appears $4k-1$ times ( because
there is $\frac{p-1}{2}=4k-1$ quadratic residues ). So that
$e_{1}e_{2}=(4k-1)+(2k-1)e_{1}+(2k-1)e_{2}=(2k-1)(1+e_{2}+e_{2})+2k$.
\end{proof}

For $p\equiv 1(\text{mod} \quad 8)$, we have the following Lemma \ref{Z2mQR:Lemma:e1^2e2^2e1*e2p1} too.
\begin{lemma}\label{Z2mQR:Lemma:e1^2e2^2e1*e2p1}
If $p\equiv 1(\text{mod} \quad 8)$, namely $p=8k+1$, then in $R_{p}$:
$e_{1}^{2}=(2k-1)e_{1}+2ke_{2}+4k$,
$e_{2}^{2}=(2k-1)e_{2}+2ke_{1}+4k$,
$e_{1}e_{2}=(2k-1)(e_{1}+e_{2})$.
\end{lemma}
It's proof is similar to the proof of Lemma \ref{Z2mQR:Lemma:e1^2e2^2e1*e2p-1} by Theorem 8.3(3)(4) [6] and Lemma \ref{Z2mQR:Lemma:iaddjp1}. Notice that $-1$ is a quadratic residue modulo $p$ in this case.

\section{Idempotent Generators of $Z_{2^{m}}$-QR codes}

Now we calculate the idempotents of $Z_{2^{m}}$-QR codes. We know these idempotent generators of quadratic residue codes over $Z_{4},Z_{8},Z_{16},Z_{2^{\infty}}$ are all linear combination of $e_{1},e_{2}$ and 1 [2,3,4,10]. And the multiplication of $e_{i}e_{j} (i,j=1\mbox{ or }2)$ is relevant with addition of
monomial of $e_{1}$ and $e_{2}$. Thus, we can assume that the idempotent generators of quadratic residue codes over $R_{p}\quad (p\equiv \pm 1(\text{mod} \quad 8))$ where $p$ can be satisfied by $p \equiv \pm 1 (\text{mod} \quad 2^{m})(m\geq 3)$ or $p \equiv \pm (8k-1) (\text{mod} \quad 2^{m})(1\leq k \leq 2^{m-3}-1,m\geq 4)$ as following
$$
\alpha +\beta e_{1}+\gamma e_{2} \quad(\alpha,\beta,\gamma\in
Z_{2^{m}}).
$$

By the definition of idempotent, $\alpha+\beta e_{1}+\gamma e_{2}$
over $R_{p}$ must satisfy
\begin{equation}\label{Z2mQR:Equation:idemmodxp-1}
(\alpha+\beta e_{1}+\gamma e_{2})^{2}\equiv (\alpha+\beta
e_{1}+\gamma e_{2})(\text{mod} \quad x^{p}-1),
\end{equation}
and
\begin{equation}\label{Z2mQR:Equation:idemsquare}
(\alpha+\beta e_{1}+\gamma e_{2})^{2}=\alpha^{2}+\beta^{2}
e_{1}^{2}+\gamma^{2}e_{2}^{2}+2\alpha \beta e_{1}+2\alpha \gamma
e_{2}+2\beta \gamma e_{1}e_{2}.
\end{equation}

If $p\equiv -1(\text{mod} \quad 8)$, namely, $p=8k-1$, by Lemma \ref{Z2mQR:Lemma:e1^2e2^2e1*e2p-1} and
equation \eqref{Z2mQR:Equation:idemmodxp-1}\eqref{Z2mQR:Equation:idemsquare}, then
\begin{equation*}
\begin{split}
(\alpha+\beta e_{1}+\gamma e_{2})^{2} & = [\alpha^{2}+2\beta
\gamma (4k-1)]+ [\beta^{2}(2k-1)+2k\gamma^{2}+2\alpha \beta \\
& \quad +2\beta
\gamma(2k-1)]e_{1}+[\gamma^{2}(2k-1)+2k\beta^{2}+2\alpha \gamma
\\
& \quad +2\beta \gamma(2k-1)]e_{2}.
\end{split}
\end{equation*}
And it is a idempotent, so
\begin{equation} \label{Z2mQR:Equation:idemabrp-1}
\left\{\begin{array}{l}
\alpha^{2}+2\beta \gamma (4k-1)\equiv \alpha (\text{mod} \quad 2^{m}) \\
\beta^{2}(2k-1)+2k\gamma^{2}+2\alpha \beta+2\beta \gamma(2k-1)\equiv \beta (\text{mod} \quad 2^{m}) \\
\gamma^{2}(2k-1)+2k\beta^{2}+2\alpha \gamma+2\beta
\gamma(2k-1)\equiv \gamma (\text{mod} \quad 2^{m}).
\end{array}
\right.
\end{equation}

If $p\equiv 1(\text{mod} \quad 8)$, namely, $p=8k+1$, by Lemma \ref{Z2mQR:Lemma:e1^2e2^2e1*e2p1} and
equation \eqref{Z2mQR:Equation:idemmodxp-1}\eqref{Z2mQR:Equation:idemsquare}, then
\begin{equation*}
\begin{split}
(\alpha+\beta e_{1}+\gamma e_{2})^{2} &=
[\alpha^{2}+4k\beta^{2}+4k\gamma^{2}]+
[\beta^{2}(2k-1)+2k\gamma^{2}+2\alpha \beta+2\beta \gamma \cdot
2k]e_{1} \\
& \quad +[\gamma^{2}(2k-1)+2k\beta^{2}+2\alpha \gamma+2\beta
\gamma \cdot 2k]e_{2}.
\end{split}
\end{equation*}
Ditto, it is a idempotent, so
\begin{equation}\label{Z2mQR:Equation:idemabrp1}
\left\{\begin{array}{l}
\alpha^{2}+4k\beta^{2}+4k\gamma^{2}\equiv \alpha (\text{mod} \quad 2^{m}) \\
\beta^{2}(2k-1)+2k\gamma^{2}+2\alpha \beta+2\beta \gamma \cdot 2k \equiv \beta (\text{mod} \quad 2^{m}) \\
\gamma^{2}(2k-1)+2k\beta^{2}+2\alpha \gamma+2\beta \gamma \cdot 2k
\equiv \gamma(\text{mod} \quad 2^{m}).
\end{array}
\right.
\end{equation}

From \eqref{Z2mQR:Equation:idemabrp-1} and \eqref{Z2mQR:Equation:idemabrp1}, we can obviously see the equivalence of
$\beta$ and $\gamma$.
\begin{theorem}\label{Z2mQR:Theorem:idembrrb}
If $\alpha+\beta e_{1}+\gamma e_{2}$ is idempotent over $R_{p}$,
then $\alpha+\gamma e_{1}+\beta e_{2}$ is also idempotent over $R_{p}$.
\end{theorem}

According to the last two equations of \eqref{Z2mQR:Equation:idemabrp-1} and \eqref{Z2mQR:Equation:idemabrp1}, we get
\begin{equation}
\begin{split}
\gamma^{2}-\beta^{2}+2\alpha(\beta-\gamma)\equiv (\beta -\gamma) (\text{mod} \quad 2^{m}) \\
\Rightarrow \quad(\beta-\gamma)[2\alpha-(\beta+\gamma)]\equiv (\beta-\gamma)(\text{mod}\quad 2^{m})\\
\Rightarrow \quad(\beta-\gamma)[2\alpha-(\beta+\gamma)-1]\equiv 0(\text{mod} \quad 2^{m}).
\end{split}
\end{equation}
Without lost of generality, we assume that $(\beta-\gamma,2^{m})\neq 0$.
Then the coefficients $\alpha,\beta,\gamma$ of idempotents maybe exist by the case is $2\alpha-(\beta+\gamma)\equiv 1(\text{mod}\quad 2^{m})$. It has been studied in [15] when $p\equiv \pm 1 (\text{mod}\quad 2^{m})$. We will discuss the general case of $p\equiv \pm (8k-1) (\text{mod} \quad 2^{m})$ in this paper.

\begin{theorem}\label{Z2mQR:Theorem:idemabr}
The idempotents $\alpha+\beta e_{1}+\gamma e_{2}$ of quadratic residue codes over $R_{p}$ satisfy that
$2\alpha-(\beta+\gamma) \equiv 1(\text{mod}\quad 2^{m})$.
\end{theorem}
\begin{proof}
Calderank and Sloane got the idempotents $\alpha^{'}+\beta^{'} e_{1}+\gamma^{'} e_{2}$ of quadratic residue codes of length $p$ over $Z_{2^{\infty}}$ [5] were
$$
\frac{p+1}{2p}+\frac{1+\sqrt{-p}}{2p}e_{1}+\frac{1-\sqrt{-p}}{2p}e_{2}
$$
and
$$
\frac{p-1}{2p}+\frac{-1+\sqrt{-p}}{2p}e_{1}+\frac{-1-\sqrt{-p}}{2p}e_{2},
$$
where the coefficients $\alpha^{'},\beta^{'},\gamma^{'}$ are the 2-$adic$ numbers. Then we can get the idempotents $\alpha+\beta e_{1}+\gamma e_{2}$ over $Z_{2^{m}}$ satisfy that
\[
2\alpha
-(\beta+\gamma)=2\times\frac{p+1}{2p}-(\frac{1+\sqrt{-p}}{2p}+\frac{1-\sqrt{-p}}{2p})\equiv
1(\text{mod} \quad 2^{m})
\]
and
\[
2\alpha
-(\beta+\gamma)=2\times\frac{p-1}{2p}-(\frac{-1+\sqrt{-p}}{2p}+\frac{-1-\sqrt{-p}}{2p})\equiv
1(\text{mod} \quad 2^{m}).
\]
So we can conclude $2\alpha-(\beta+\gamma) \equiv 1(\text{mod}\quad 2^{m})$.
\end{proof}

We know that the all-1 vector $1+e_{1}+e_{2}$ denoted by $h$, is an idempotent in the binary case. In the ring of $R_{p}$,
we get $x^{p}=1\Rightarrow x^{p+t}=x^{t}\Rightarrow x^{i}h=h$ (for
$\forall i=0,1,\cdots,p-1$), then
\begin{equation*}
\begin{split}
h^{2} &=(1+e_{1}+e_{2})h \\
& =h+e_{1}h+e_{2}h \\
& =h+\frac{p-1}{2}h+\frac{p-1}{2}h \\
& =ph
\end{split}
\end{equation*}

\begin{theorem}
If $p\equiv \pm (8k-1)(\text{mod} \quad 2^{m})\quad (1\leq k \leq
 2^{m-3}-1, m\geq 4)$ and $\alpha+\beta e_{1}+\gamma e_{2}$ is an idempotent satisfied by $2\alpha-(\beta+\gamma) \equiv 1(\text{mod}\quad 2^{m})$ over
$R_{p}$, then there is $\beta+\gamma \equiv \pm p(\text{mod} \quad 2^{m})$. And $\alpha+\beta e_{1}+\gamma e_{2}-(8k-1)h$ is an idempotent over
 $R_{p}$ when $\beta+\gamma \equiv (8k-1)(\text{mod} \quad 2^{m})$; $\alpha+\beta e_{1}+\gamma e_{2}+(8k-1)h$ is an idempotent over
 $R_{p}$ when $\beta+\gamma \equiv -(8k-1)(\text{mod} \quad 2^{m})$.
\end{theorem}
\begin{proof}
The idempotents over $Z_{2^{\infty}}$ are $\alpha^{'}+\beta^{'}e_{1}+\gamma^{'}e_{2}$ [5], where $\alpha^{'},\beta^{'},\gamma^{'}$ are 2-$adic$ numbers
$$
\alpha^{'}=\frac{p+1}{2p},\quad \beta^{'}=\frac{1+\sqrt{-p}}{2p},\quad \gamma^{'}=\frac{1-\sqrt{-p}}{2p}
$$
and
$$
\alpha^{'}=\frac{p-1}{2p},\quad \beta^{'}=\frac{-1+\sqrt{-p}}{2p},\quad \gamma^{'}=\frac{-1-\sqrt{-p}}{2p}.
$$
Obviously, we get $\beta^{'}+\gamma^{'}=\pm \frac{1}{p}$ and the idempotent $\alpha+\beta e_{1}+\gamma e_{2}$ over $Z_{2^{m}}$ satisfied by $\beta+\gamma \equiv (\beta^{'}+\gamma^{'})(\text{mod} \quad 2^{m})$.

At first, we consider the case of $p\equiv (8k-1)(\text{mod} \quad 2^{m})\quad (1\leq k \leq
2^{m-3}-1)$. The 2-$adic$ representation of $p$ modulo $2^{m}$ is (see Appendix \ref{Z2mQR:Appendix:p8k-1})
\begin{equation}\label{Z2mQR:Equation:2-adicp8k-1}
1+1\cdot 2+1\cdot 2^{2}+n_{1}\cdot 2^{3}+n_{2}\cdot 2^{4}+\cdots+n_{m-3}\cdot 2^{m-1},
\end{equation}
where $n_{1},n_{2},\cdots,n_{m-3}$ is 0 or 1. The 2-$adic$ representation of $\frac{1}{p}$ modulo $2^{m}$ is also like equation \eqref{Z2mQR:Equation:2-adicp8k-1} (see Appendix \ref{Z2mQR:Appendix:1/p8k-1}). The 2-$adic$ representation of $-p$ modulo $2^{m}$ is (see Appendix \ref{Z2mQR:Appendix:-p8k-1})
\begin{equation}\label{Z2mQR:Equation:2-adic1/p}
1+n_{1}\cdot 2^{3}+n_{2}\cdot 2^{4}+\cdots+n_{m-3}\cdot 2^{m-1},
\end{equation}
where $n_{1},n_{2},\cdots,n_{m-3}$ is 0 or 1. At the same time, the 2-$adic$ representation of $-\frac{1}{p}$ is also
like equation \eqref{Z2mQR:Equation:2-adic1/p} (see Appendix \ref{Z2mQR:Appendix:-1/p8k-1}). Then we get $ \pm p \equiv \pm \frac{1}{p}(\text{mod} \quad 2^{m})$ when $p\equiv (8k-1)(\text{mod} \quad 2^{m})$. So we conclude $\beta+\gamma\equiv \beta^{'}+\gamma^{'}\equiv \pm p (\text{mod} \quad 2^{m})$. Similarly, we can get the same conclusion when $p\equiv \pm -(8k-1)(\text{mod} \quad 2^{m})\quad (1\leq k \leq 2^{m-3}-1)$.

We have known that $h^{2}=ph$ and $\alpha+\beta e_{1}+\gamma e_{2}$ is
idempotent satisfied by $2\alpha-(\beta+\gamma) \equiv 1(\text{mod}\quad 2^{m})$.

If $\beta+\gamma\equiv (8k-1)(\text{mod} \quad 2^{m})$, then

$(\alpha+\beta e_{1}+\gamma e_{2}-(8k-1)h)^{2}$ \vspace{-0.3cm}
\begin{equation}\label{Z2mQR:Equation:beta+gamma8k-1}
\begin{split}
& =(\alpha +\beta e_{1}+\gamma e_{2})^{2}-2(8k-1)h(\alpha+\beta e_{1}+\gamma e_{2})+(8k-1)^{2}h^{2} \\
& =(\alpha + \beta e_{1}+\gamma e_{2})-(8k-1)h(2\alpha+2\beta\cdot \frac{p-1}{2}+2\gamma\cdot \frac{p-1}{2})+(8k-1)^{2}ph \\
& =(\alpha+\beta e_{1}+\gamma e_{2})-(8k-1)h[2\alpha-(\beta+\gamma)+p(\beta+ \gamma -8k+1)] \\
& =\alpha+\beta e_{1}+\gamma e_{2}-(8k-1)h.
\end{split}
\end{equation}
So $\alpha+\beta e_{1}+\gamma e_{2}-(8k-1)h$ is an idempotent by equation \eqref{Z2mQR:Equation:beta+gamma8k-1}.

If $\beta+\gamma\equiv -(8k-1)(\text{mod} \quad 2^{m})$, then

$\alpha+\beta e_{1}+\gamma e_{2}+(8k-1)h)^{2}$ \vspace{-0.3cm}
\begin{equation}\label{Z2mQR:Equation:beta+gamma-(8k-1)}
\begin{split}
& = (\alpha +\beta e_{1}+\gamma e_{2})^{2}+2(8k-1)h(\alpha+\beta e_{1}+\gamma e_{2})+(8k-1)^{2}h^{2} \\
& =(\alpha + \beta e_{1}+\gamma e_{2})+(8k-1)h(2\alpha+2\beta\cdot \frac{p-1}{2}+2\gamma\cdot \frac{p-1}{2})+(8k-1)^{2}ph\\
& =(\alpha+\beta e_{1}+\gamma e_{2})+(8k-1)h[2\alpha-(\beta+\gamma)+p(\beta+ \gamma +8k-1)] \\
& =\alpha+\beta e_{1}+\gamma e_{2}+(8k-1)h.
\end{split}
\end{equation}
Then $\alpha+\beta e_{1}+\gamma e_{2}+(8k-1)h$is an idempotent by equation \eqref{Z2mQR:Equation:beta+gamma-(8k-1)}.
\end{proof}

\section{Quadratic Residue Codes over $Z_{2^{m}}$}

We have found some idempotents over $R_{p}$ in section III. Then we can use these idempotents as generators to
define some quadratic residue codes over $R_{p}$ and discuss their properties.
\begin{definition}\label{Z2mQR:Definition:QRp8k-1}
Let $p\equiv \pm 1(\text{mod} \quad 8)$ be an odd prime and $p\equiv \pm (8k-1)(\text{mod} \quad 2^{m})(1\leq k\leq 2^{m-3}-1, m\geq 4)$. Here $\alpha+\beta e_{1}+\gamma e_{2} (\alpha,\beta,\gamma\in Z_{2^{m}}$) is the idempotent over $R_{p}$ satisfied by $2\alpha-(\beta+\gamma) \equiv 1(\text{mod}\quad 2^{m})$. The quadratic residue codes $Q_{2^{m}},Q_{2^{m}}^{'},N_{2^{m}},N_{2^{m}}^{'}$ of length $p$ are defined by the following case.
\vspace{-0.3cm}
\begin{enumerate}
 \item[(1)] Suppose that $p\equiv (8k-1)(\text{mod} \quad 2^{m})$.\vspace{-0.3cm}
  \begin{enumerate}
   \item[(1.1)] Let $Q_{2^{m}}=(\alpha+\beta e_{1}+\gamma e_{2}-(8k-1)h)\mbox{, }Q_{2^{m}}^{'}=(\alpha+\beta e_{1}+\gamma e_{2}) \mbox{, }N_{2^{m}}=(\alpha+\gamma e_{1}+\beta e_{2}-(8k-1)h)\mbox{, }N_{2^{m}}^{'}=(\alpha+\gamma e_{1}+ \beta e_{2})$ if $\beta+\gamma\equiv p(\text{mod} \quad 2^{m})$ and $ p^{2}\equiv -1(\text{mod} \quad 2^{m})$;
   \item[(1.2)] Let $Q_{2^{m}}=(\alpha+\beta e_{1}+\gamma e_{2})\mbox{, }Q_{2^{m}}^{'}=(\alpha+\beta e_{1}+\gamma e_{2}+(8k-1)h)\mbox{, }N_{2^{m}}=(\alpha+\gamma e_{1}+\beta e_{2})\mbox{, }N_{2^{m}}^{'}=(\alpha+\gamma e_{1}+ \beta e_{2}+(8k-1)h)$ if $\beta+\gamma\equiv -p(\text{mod} \quad 2^{m})$ and $ p^{2}\equiv 1(\text{mod} \quad 2^{m})$;
   \end{enumerate}   \vspace{-0.3cm}
 \item[(2)] Suppose that $p\equiv -(8k-1)(\text{mod} \quad 2^{m})$. \vspace{-0.3cm}
 \begin{enumerate}
   \item[(2.1)] Let $Q_{2^{m}}=(\alpha+\beta e_{1}+\gamma e_{2})\mbox{, }Q_{2^{m}}^{'}=(\alpha+\beta e_{1}+\gamma e_{2}-(8k-1)h) \mbox{, }N_{2^{m}}=(\alpha+\gamma e_{1}+\beta e_{2})\mbox{, }N_{2^{m}}^{'}=(\alpha+\gamma e_{1}+ \beta e_{2}-(8k-1)h)$ if $\beta+\gamma\equiv -p(\text{mod} \quad 2^{m})$ and $ p^{2}\equiv 1(\text{mod} \quad 2^{m})$;
   \item[(2.2)] Let $Q_{2^{m}}=(\alpha+\beta e_{1}+\gamma e_{2}+(8k-1)h) \mbox{,}Q_{2^{m}}^{'}=(\alpha+\beta e_{1}+\gamma e_{2})\mbox{, }N_{2^{m}}=(\alpha+\gamma e_{1}+\beta e_{2}+(8k-1)h)\mbox{, }N_{2^{m}}^{'}=(\alpha+\gamma e_{1}+ \beta e_{2})$ if $\beta+\gamma\equiv p(\text{mod} \quad 2^{m})$ and $ p^{2}\equiv -1(\text{mod} \quad 2^{m})$.
 \end{enumerate}
\end{enumerate}
\end{definition}

These quadratic residue codes over $R_{p}$ have many similar properties with binary quadratic residue codes. The
following theorems show us that is true.

\begin{theorem} \label{Z2mQR:Theorem:Def1.1}
 The quadratic residue cods $Q_{2^{m}},Q_{2^{m}}^{'},N_{2^{m}},N_{2^{m}}^{'}$ over $R_{p}$ are defined by Definition \ref{Z2mQR:Definition:QRp8k-1} (1.1). Then they have properties as following.
\vspace{-0.3cm}
\begin{enumerate}
   \item[(1)]
   $Q_{2^{m}}$ and $N_{2^{m}}$ are equivalent. $Q_{2^{m}}^{'}$ and $N_{2^{m}}^{'}$ are equivalent.\vspace{-0.3cm}
   \item[(2)]
   $Q_{2^{m}}\bigcap N_{2^{m}}=(-(8k-1)h)$ and $Q_{2^{m}}+N_{2^{m}}=R_{p}$.\vspace{-0.3cm}
   \item[(3)]
   $|Q_{2^{m}}|=2^{m\cdot(p+1)/2}=|N_{2^{m}}|$.\vspace{-0.3cm}
   \item[(4)]
   $Q_{2^{m}}=Q_{2^{m}}^{'}+(-(8k-1)h)$ and $N_{2^{m}}=N_{2^{m}}^{'}+(-(8k-1)h)$.\vspace{-0.3cm}
   \item[(5)]
   $|Q_{2^{m}}^{'}|=2^{m\cdot(p-1)/2}=|N_{2^{m}}^{'}|$.\vspace{-0.3cm}
   \item[(6)]
   $Q_{2^{m}}^{'},N_{2^{m}}^{'}$ are both self-orthogonal. $Q_{2^{m}}^{\perp}=Q_{2^{m}}^{'},N_{2^{m}}^{\perp}=N_{2^{m}}^{'}$.\vspace{-0.3cm}
\end{enumerate}
\end{theorem}
\begin{proof}
(1) Let the map $\mu_{a}$ be defined as following
$$ \mu_{a}:i\rightarrow ai(\text{mod} \quad p) \quad \text{for any nonzero} \quad a \in GF(p).$$
If $f=\sum_{i=0}^{p-1}c_{i}x^{i}\in R_{p}$, then $\mu_{a}(f)=\sum_{i=0}^{p-1}c_{\mu_{a}(i)}x^{\mu_{a}(i)}$. Here $\mu_{a}(i)\equiv ai(\text{mod} \quad p)$. It is not hard to show that $\mu_{a}(fg)=\mu_{a}(f)\mu_{a}(g)$ for $f$ and $g$ that are polynomials over $R_{p}$. Let $n$ be the element in the set $N$. For any $r\in Q\Rightarrow nr\in N$. Then $\mu_{n}(e_{1})=e_{2},\mu_{n}(e_{2})=e_{1}\Rightarrow \mu_{n}(\alpha+\beta e_{1}+\gamma e_{2})=\alpha+\gamma e_{1}+\beta e_{2}$. Similarly, $\mu_{n}(\alpha+\beta e_{1}+\gamma e_{2}+(8k-1)h)=\alpha+\gamma e_{1}+\beta e_{2}+(8k-1)h$. It shows that $Q_{2^{m}}$ and $N_{2^{m}}$ are equivalent. Similarly, $Q_{2^{m}}^{'}$ and $N_{2^{m}}^{'}$ can also be proved to be equivalent.

(2) Since $\beta +\gamma\equiv p\equiv (8k-1)(\text{mod} \quad 2^{m})$ and $2\alpha-(\beta+\gamma) \equiv 1(\text{mod}\quad 2^{m})$,
we get $2\alpha \equiv 8k(\text{mod} \quad 2^{m})$. $Q_{2^{m}}$, $Q_{2^{m}}^{'}$, $N_{2^{m}}$, $N_{2^{m}}^{'}$ are defined by the first part of Definition \ref{Z2mQR:Definition:QRp8k-1} (1.1). Because
\vspace{-0.3cm}
\begin{equation*}
\begin{split}
-(8k-1)h &=-(8k-1)-(8k-1)e_{1}-(8k-1)e_{2} \\
&=-1+[\alpha+\beta e_{1}+\gamma e_{2}-(8k-1)h]+[\alpha+\gamma e_{1}+\beta e_{2}-(8k-1)h]
\end{split}
\end{equation*}
$\Rightarrow [\alpha+\beta e_{1}+\gamma e_{2}-(8k-1)h][-(8k-1)h]$
\begin{equation}\label{Z2mQR:Equation:abr-8k-11}
\begin{split}
&=[\alpha+\beta e_{1}+\gamma e_{2}-(8k-1)h]\\
&\qquad \{-1+[\alpha+\beta e_{1}+\gamma e_{2}-(8k-1)h]+[\alpha+\gamma e_{1}+\beta e_{2}-(8k-1)h]\}\\
&=-[\alpha+\beta e_{1}+\gamma e_{2}-(8k-1)h]+[\alpha+\gamma e_{1}+\beta e_{2}-(8k-1)h]^{2}\\
&\qquad  +[\alpha+\beta e_{1}+\gamma e_{2}-(8k-1)h][\alpha+\gamma e_{1}+\beta e_{2}-(8k-1)h]\\
&=[\alpha+\beta e_{1}+\gamma e_{2}-(8k-1)h][\alpha+\gamma e_{1}+\beta e_{2}-(8k-1)h]
\end{split}
\end{equation}
From equation \eqref{Z2mQR:Equation:abr-8k-11}, $[\alpha+\beta e_{1}+\gamma e_{2}-(8k-1)h][-(8k-1)h]=[\alpha+\beta e_{1}+\gamma e_{2}-(8k-1)h][\alpha+\gamma e_{1}+\beta e_{2}-(8k-1)h]$. On the other hand,

$[\alpha+\beta e_{1}+\gamma e_{2}-(8k-1)h][-(8k-1)h]$
\begin{equation}\label{Z2mQR:Equation:abr-8k-12}
\begin{split}
&=[\frac{2\alpha-(\beta+\gamma)+p(\beta+\gamma)}{2}-(8k-1)p][-(8k-1)h]\\
&=[\frac{1-p^{2}}{2}][-(8k-1)h]\quad (\because p^{2}\equiv -1(\text{mod} \quad 2^{m}))\\
&=-(8k-1)h.
\end{split}
\end{equation}
From equation \eqref{Z2mQR:Equation:abr-8k-11} and \eqref{Z2mQR:Equation:abr-8k-12}, then \vspace{-0.3cm} $$[\alpha+\beta e_{1}+\gamma e_{2}-(8k-1)h][\alpha+\gamma
e_{1}+\beta e_{2}-(8k-1)h]=-(8k-1)h.$$

\vspace{-0.3cm} So we get $Q_{2^{m}}\bigcap N_{2^{m}}$ has idempotent generator $-(8k-1)h$.
$|Q_{2^{m}}\bigcap N_{2^{m}}|=|-(8k-1)h|=2^{m}$.
$Q_{2^{m}}+N_{2^{m}}$ has idempotent generator $[\alpha+\beta
e_{1}+\gamma e_{2}-(8k-1)h]+[\alpha+\gamma e_{1}+\beta
e_{2}-(8k-1)h]-[\alpha+\beta e_{1}+\gamma
e_{2}-(8k-1)h][\alpha+\gamma e_{1}+\beta e_{2}-(8k-1)h]=2\alpha-(8k-1)=1$. Then $Q_{2^{m}}+N_{2^{m}}=R_{p}$ and $|Q_{2^{m}}+N_{2^{m}}|=(2^{m})^{p}=2^{mp}$.

(3) We have known that $Q_{2^{m}}$ and $N_{2^{m}}$ are equivalent from (1), so $|Q_{2^{m}}|=|N_{2^{m}}|$. Then $|Q_{2^{m}}+N_{2^{m}}|=|Q_{2^{m}}|\cdot|N_{2^{m}}|/|Q_{2^{m}}\bigcap
N_{2^{m}}|\Rightarrow |Q_{2^{m}}|=|N_{2^{m}}|=2^{m(p+1)/2}$.

(4) Since $\beta+\gamma\equiv p\equiv (8k-1)(\text{mod} \quad 2^{m})$ and
$\frac{2\alpha-(\beta+\gamma)+p(\beta+\gamma)}{2}\equiv 0(\text{mod}
\quad 2^{m})(\because p^{2}\equiv -1(\text{mod} \quad 2^{m}))$, we get \vspace{-0.3cm}
\begin{equation*}
\begin{split}
(\alpha+\beta e_{1}+\gamma e_{2})[-(8k-1)h]
&=\frac{2\alpha-(\beta+\gamma)+p(\beta+\gamma)}{2}[-(8k-1)h] \\
&=0.
\end{split}
\end{equation*}
Then the fact is $Q_{2^{m}}^{'}\bigcap (-(8k-1)h)=0$ and
$Q_{2^{m}}^{'}+(-(8k-1)h)$ has idempotent generator $(\alpha+\beta
e_{1}+\gamma e_{2})+[-(8k-1)h]-(\alpha+\beta e_{1}+\gamma
e_{2})[-(8k-1)h]=\alpha +\beta e_{2}+\gamma e_{2}-(8k-1)h$. Hence,
$Q_{2^{m}}^{'}+(-(8k-1)h) =(\alpha+\beta e_{1}+ \gamma
e_{2}-(8k-1)h)=Q_{2^{m}}$. Similarly, it can be proved to
$N_{2^{m}}^{'}+(-(8k-1)h)=N_{2^{m}}$.

(5) It is true that $2^{m(p+1)/2}=|Q_{2^{m}}|=|Q_{2^{m}}^{'}+(-(8k-1)h)|=|Q_{2^{m}}^{'}||(-(8k-1)h)|=|Q_{2^{m}}^{'}|\cdot 2^{m}$. So
$|Q_{2^{m}}^{'}|=2^{m(p+1)/2-m}=2^{m(p-1)/2}$. Similarly,
$|N_{2^{m}}^{'}|=2^{m(p-1)/2}$.

(6) Because $p\equiv (8k-1)(\text{mod} \quad 2^{m})\Rightarrow p\equiv -1 (\text{mod}\quad 8)$, $-1$ is a quadratic
non-residue modulo $p$ and $e_{1}(x^{-1})=e_{2}(x),e_{2}(x^{-1})=e_{1}(x)$. Since $\beta+\gamma \equiv (8k-1)(\text{mod} \quad 2^{m})\Rightarrow
(8k-1)-\gamma\equiv \beta(\text{mod} \quad 2^{m}), (8k-1)-\beta\equiv \gamma(\text{mod} \quad 2^{m})$. And $2\alpha \equiv 8k (\text{mod} \quad 2^{m})\Rightarrow 1-\alpha+ (8k-1)\equiv \alpha (\text{mod} \quad 2^{m})$ for $\beta+\gamma \equiv (8k-1)(\text{mod} \quad 2^{m})$ and $2\alpha-(\beta+\gamma)\equiv 1(\text{mod} \quad 2^{m})$. Obviously, $Q_{2^{m}}^{\perp}$ has idempotent generator \vspace{-0.3cm}
$$ 1-[\alpha+\beta e_{1}(x^{-1})+\gamma e_{2}(x^{-1})-(8k-1)-(8k-1)e_{1}(x^{-1})-(8k-1)e_{2}(x^{-1})]$$
\begin{equation}\label{Z2mQR:Equation:Q2mvertical}
\begin{split}
&=1-[\alpha+\beta e_{2}+\gamma e_{1}-(8k-1)-(8k-1)e_{2}-(8k-1)e_{1}]\\
&=[1-\alpha+(8k-1)]+[(8k-1)-\gamma]e_{1}+ [(8k-1)-\beta]e_{2} \\
&=\alpha+\beta e_{1}+\gamma e_{2}.
\end{split}
\end{equation}

From the equation \eqref{Z2mQR:Equation:Q2mvertical}, we get $ Q_{2^{m}}^{\perp}=Q_{2^{m}}^{'}\subseteq
Q_{2^{m}}=Q_{2^{m}}^{'\perp}$. So $Q_{2^{m}}^{'}$ is self-orthogonal. Similarly, $N_{2^{m}}^{\perp}=N_{2^{m}}^{'}$ and
$N_{2^{m}}^{'}$ is also self-orthogonal.
\end{proof}

\begin{theorem}  \label{Z2mQR:Theorem:Def1.2}
$Q_{2^{m}}$, $Q_{2^{m}}^{'}$, $N_{2^{m}}$, $N_{2^{m}}^{'}$ are
defined by Definition  \ref{Z2mQR:Definition:QRp8k-1} (1.2). Then they have properties as following.
\vspace{-0.3cm}
\begin{enumerate}
   \item[(1)]
   $Q_{2^{m}}$ and $N_{2^{m}}$ are equivalent. $Q_{2^{m}}^{'}$ and $N_{2^{m}}^{'}$ are also equivalent. \vspace{-0.3cm}
   \item[(2)]
   $Q_{2^{m}}^{'}\bigcap N_{2^{m}}^{'}=((8k-1)h)$ and  $Q_{2^{m}}^{'}+N_{2^{m}}^{'}=R_{p}$. \vspace{-0.3cm}
   \item[(3)]
   $|Q_{2^{m}}^{'}|=2^{m\cdot(p+1)/2}=|N_{2^{m}}^{'}|$. \vspace{-0.3cm}
   \item[(4)]
   $Q_{2^{m}}^{'}=Q_{2^{m}}+((8k-1)h)$ and $N_{2^{m}}^{'}=N_{2^{m}}+((8k-1)h)$. \vspace{-0.3cm}
   \item[(5)]
   $|Q_{2^{m}}|=2^{m\cdot(p-1)/2}=|N_{2^{m}}|$. \vspace{-0.3cm}
   \item[(6)]
   $Q_{2^{m}},N_{2^{m}}$ are both self-orthogonal. $Q_{2^{m}}^{'\perp}=Q_{2^{m}},N_{2^{m}}^{'\perp}=N_{2^{m}}$. \vspace{-0.3cm}
\end{enumerate}
\end{theorem}

\begin{theorem}  \label{Z2mQR:Theorem:Def1.3}
$Q_{2^{m}}$, $Q_{2^{m}}^{'}$, $N_{2^{m}}$, $N_{2^{m}}^{'}$ are
defined by Definition  \ref{Z2mQR:Definition:QRp8k-1} (2.1). Then they have properties as following.
\vspace{-0.3cm}
\begin{enumerate}
   \item[(1)]
   $Q_{2^{m}}$ and $N_{2^{m}}$ are equivalent. $Q_{2^{m}}^{'}$ and $N_{2^{m}}^{'}$ are also equivalent. \vspace{-0.3cm}
   \item[(2)]
   $Q_{2^{m}}\bigcap N_{2^{m}}=(-(8k-1)h)$ and  $Q_{2^{m}}+N_{2^{m}}=R_{p}$. \vspace{-0.3cm}
   \item[(3)]
   $|Q_{2^{m}}|=2^{m\cdot(p+1)/2}=|N_{2^{m}}|$. \vspace{-0.3cm}
   \item[(4)]
   $Q_{2^{m}}=Q_{2^{m}}^{'}+(-(8k-1)h)$ and $N_{2^{m}}=N_{2^{m}}^{'}+(-(8k-1)h)$. \vspace{-0.3cm}
   \item[(5)]
   $|Q_{2^{m}}^{'}|=2^{m\cdot(p-1)/2}=|N_{2^{m}}^{'}|$. \vspace{-0.3cm}
   \item[(6)]
   $Q_{2^{m}}^{'},N_{2^{m}}^{'}$ are both self-orthogonal. $Q_{2^{m}}^{\perp}=Q_{2^{m}}^{'},N_{2^{m}}^{\perp}=N_{2^{m}}^{'}$. \vspace{-0.3cm}
\end{enumerate}
\end{theorem}

\begin{theorem}   \label{Z2mQR:Theorem:Def1.4}
$Q_{2^{m}}$, $Q_{2^{m}}^{'}$, $N_{2^{m}}$, $N_{2^{m}}^{'}$ are
defined by Definition  \ref{Z2mQR:Definition:QRp8k-1} (2.2). Then they have properties as following.
\vspace{-0.3cm}
\begin{enumerate}
   \item[(1)]
   $Q_{2^{m}}$ and $N_{2^{m}}$ are equivalent. $Q_{2^{m}}^{'}$ and $N_{2^{m}}^{'}$ are also equivalent. \vspace{-0.3cm}
   \item[(2)]
   $Q_{2^{m}}^{'}\bigcap N_{2^{m}}^{'}=((8k-1)h)$ and  $Q_{2^{m}}^{'}+N_{2^{m}}^{'}=R_{p}$. \vspace{-0.3cm}
   \item[(3)]
   $|Q_{2^{m}}^{'}|=2^{m\cdot(p+1)/2}=|N_{2^{m}}^{'}|$. \vspace{-0.3cm}
   \item[(4)]
   $Q_{2^{m}}^{'}=Q_{2^{m}}+((8k-1)h)$ and $N_{2^{m}}^{'}=N_{2^{m}}+((8k-1)h)$. \vspace{-0.3cm}
   \item[(5)]
   $|Q_{2^{m}}|=2^{m\cdot(p-1)/2}=|N_{2^{m}}|$. \vspace{-0.3cm}
   \item[(6)]
   $Q_{2^{m}},N_{2^{m}}$ are both self-orthogonal. $Q_{2^{m}}^{'\perp}=Q_{2^{m}},N_{2^{m}}^{'\perp}=N_{2^{m}}$. \vspace{-0.3cm}
\end{enumerate}
\end{theorem}

The proof of Theorem \ref{Z2mQR:Theorem:Def1.2}, Theorem \ref{Z2mQR:Theorem:Def1.3} and Theorem \ref{Z2mQR:Theorem:Def1.4} are similar to Theorem \ref{Z2mQR:Theorem:Def1.1}.

\section{The Minimum hamming weight of Quadratic Residue Codes over $Z_{2^{m}}$}

G.H.Norton and A.Salagean got some valuable results about hamming weight of linear codes over a finite chain ring. They have
the following conclusion by their Example 4.16 based on corollary 4.3 about quadratic residue codes over ring [16].
\begin{theorem}
If quadratic residue codes over $R_{p}$ as Hensel lifts of binary quadratic residue codes, their minimum hamming weight is the same as the minimum hamming weight of the original binary codes.
\end{theorem}
Then the minimum hamming distance of the quadratic residue codes defined by Definition \ref{Z2mQR:Definition:QRp8k-1} over $R_{p}$ is same as the minimum hamming weight of binary quadratic residue codes.

\begin{definition}
The extended code of a $Z_{2^{m}}$ code $\mathcal{C}$ denoted by $\bar{\mathcal{C}}$ is the code obtained by adding an overall parity check to each codeword of $\mathcal{C}$.
\end{definition}

Recall that the code we obtain by removing a column of a generator matrix of $\mathcal{C}$ is called a punctured $\mathcal{C}$.

\begin{definition}
A vector in a $Z_{2^{m}}$-code is "even-like" if the sum of its coordinates is $0(\text{mod}\quad 2^{m})$; otherwise, it is "odd-like".
\end{definition}

Similarly, we can conjecture that all codes obtained from an extended $Z_{2^{m}}$-QR code by puncturing must be equivalent, by fact that we can get some group $G$ contained in the group of the extended quadratic residue codes is transitive (as discussion in [4] and [11]). Then we can get the following theorem.

\begin{theorem}
The minimum hamming weight vectors of a $Z_{2^{m}}$-QR $(p,\frac{p+1}{2})$ code are odd-like.
\end{theorem}

\begin{proof}
Suppose that a $Z_{2^{m}}$-QR code $Q_{2^{m}}$ has even-like minimum hamming weight vector \textbf{$x$} of minimum weight $d$. Consider the
extended quadratic residue code $\bar{Q}_{2^{m}}$, the vector \textbf{$\bar{x}$} with the same component in \textbf{$x$}, and 0 at its $\infty$. Since
\textbf{$x$} is even-like, and the vector \textbf{$\bar{x}$} is
existing in $\bar{Q}_{2^{m}}$, so $wt(\bar{x})=wt(x)=d$.

Puncture $\bar{Q}_{2^{m}}$ on a coordinate position where $\bar{x}$
has a nonzero coordinate call the punctured code $Q^{*}_{2^{m}}$, so
$Q^{*}_{2^{m}}$ is equivalent to $Q_{2^{m}}$. But $Q^{*}_{2^{m}}$ has vector in it of
weight less then $d$ (since the punctured \textbf{$x$}). This
is a contradiction that $Q^{*}_{2^{m}}$ and $Q_{2^{m}}$ has the same minimum hamming
weight. So that \textbf{$x$} is odd-like.
\end{proof}

\section{Conclusion}
Quadratic residue codes are a particularly interesting family of cyclic codes. We define such family of codes in terms of their
idempotent generators and show that these codes
also have many good properties which are analogous in many respects to properties of binary quadratic
residue codes. Such codes constructed are self-orthogonal. And we also discuss their minimum hamming weight.

\section*{Acknowledgments} \vspace{-0.3cm}
Research supported by the National Natural Science Foundation of China under Grant 61003258 and Special
Funds for Work Safety of Guangdong Province of 2010 from Administration of Work Safety
of Guangdong Province of China are gratefully acknowledged. The author especially thanks
Hoi-Kwong Lo for the hospitality during her stay at the University of Toronto.

\vspace{0.6cm}
\leftline{\textbf{\large{Reference}}} \vspace{0.1cm}
{\parindent=0pt}
\def\toto#1#2{\centerline{\hbox to 0.5cm{#1\hss}
\parbox[t]{15cm}{#2}}}

\toto{[1]}{A.R.Hammons,Jr., P.V.Kumar, A.R.Calderbank,
N.J.A.Sloane, P.Sol$\acute{e}$, \emph{The $Z_{4}$-linearity of
Kerdock,Preparata,Goethals, and related codes} [J], IEEE Trans.
Inform.Theory, 1994, \textbf{40}(2):301--319. } \vspace{0.1cm}

\toto{[2]}{A.Bonnecaze, P.Sol$\acute{e}$, A.R.Calderbank,
\emph{Quaternary Quadratic Residue Codes and Unimodular Lattices}
[J], IEEE Trans.Inform.Theory,1995,\textbf{41}(2):366--377.}\vspace{0.1cm}

\toto{[3]}{V.S.Pless, ZhongQiang Qian, \emph{Cyclic Codes and
Quadratic Residue Codes over $Z_{4}$} [J], IEEE Trans. Inform.
Theory, 1996,\textbf{42}(5):1594--1600. }\vspace{0.1cm}

\toto{[4]}{Mei Hui Chiu, Stephen S. -T, Yau Yung Yu,
\emph{$Z_{8}$-Cyclic Codes and Quadratic Residue Codes} [J],
Advances in Applied Mathematics, 2000, \textbf{25}(1):12--33.
}\vspace{0.1cm}

\toto{[5]}{A.R.Calderank, N.J.A.Sloane, \emph{Modular and p-adic
cyclic codes}[J], Des. COdes. Cryptogr,
1995,6(1):21--35}\vspace{0.1cm}

\toto{[6]}{Pramod Kanwar, \emph{Quadratic residue codes over the
integers modulo $q\sp m$}[J], Contemp.Math. 2000, \textbf{259}:
299--312. }\vspace{0.1cm}

\toto{[7]}{T.Abualrub, R.Oehmke, \emph{On the Generator of $Z_{4}$ Cyclic Codes of Length $2^{e}$} [J], IEEE Trans.Inform.Theory, 2003, \textbf{49}(9):
2126--2133. }\vspace{0.1cm}

\toto{[8]}{J.Wolfmann, \emph{Binary Images of Cyclic Codes over
$Z_{4}$} [J], IEEE Trans.Inform.Theory, 2001, \textbf{47}(5):
1773--1779. }\vspace{0.1cm}

\toto{[9]}{J.Wolfmann, \emph{Negacyclic and cyclic codes over
$Z_{4}$} [J], IEEE Trans.Inform.Theory, 1999, \textbf{45}(5):
2527--2532. }\vspace{0.1cm}

\toto{[10]}{Xiaoqing Tan, \emph{Quadratic Residue Code over $Z_{16}$} [J], Journal of Mathematical Research and Exposition, 2005, \textbf{25}(4):739-748}\vspace{0.1cm}

\toto{[11]}{Chung-Lin.Hsu, WeiLiang Luo, Stephen S.-T.Yau, and Yung Yu, \emph{Automorphism Groups of the Extended Quadratic Residue Codes over $Z_{16}$ and $Z_{32}$}[J], The Rockey Mountain Journal of Mathematics, 2009, \textbf{39}(6):1947-1991.}\vspace{0.1cm}

\toto{[12]}{V. Pless, \emph{Introduction to the theory of error-correcting codes} [M], Second edition, Wiley Interscience, 1989.}\vspace{0.1cm}

\toto{[13]}{F.J. MacWilliams, N.J.A. Sloane,\emph{Theory of error-Correcting codes} [M], North-Holland, Amsterdam, 1978.} \vspace{0.1cm}

\toto{[14]}{V.Pless, P.Sol$\acute{e}$, Z.Qian. \emph{Cyclic self
dual codes} [J]. Finite Fields and Their Application. 1997,
\textbf{3}(1):48--69. }\vspace{0.1cm}

\toto{[15]}{Xiaoqing Tan, \emph{A kind of quadratic residue codes over $Z_{2^{m}}$ and their extended codes} [J]. Journal of Jinan University (Natural Science). 2011,
\textbf{32}(3):253--257 }\vspace{0.1cm}

\toto{[16]}{G.H.Norton, A. Salagean, \emph{On the Hamming distance of linear codes over finite chain rings} [J]. IEEE Trans. Inform. Theory, 2000, \textbf{46}:1060-1067}

\newpage
\textbf{Appendix}
\appendix
\section{2-adic representation of $p$ modulo $2^{m}$ (when $p\equiv (8k-1)(\text{mod} \quad 2^{m})\quad
(1\leq k \leq 2^{m-3}-1, m\geq 4)$)} \label{Z2mQR:Appendix:p8k-1}

If we want to obtain the 2-adic representation of $p$ modulo $2^{m}$, we just need to find the 2-adic solution
of equation $x-p=0$ modulo $2^{m}$.

Firstly, solve congruence equation $$x-p\equiv 0(\text{mod} \quad 2).$$
Namely, $x-(2^{m}t+8k-1)\equiv 0(\text{mod} \quad 2)\quad (0\leq x<2)$. Then $x=1$ and $(x-p)^{'}|_{x=1}=1\neq 0(\text{mod} \quad 2)$.

Let $x=1+2y$. Consider the solution of congruence equation $$1+2y-p\equiv 0(\text{mod} \quad 2^{2}).$$ Namely, $(1+2y)-(2^{m}t+8k-1)\equiv
0(\text{mod} \quad 2^{2})\quad (0\leq y<2).$  The solution of this congruence equation is same to congruence equation
$-2^{m-1}t-4k+1+y\equiv y+1 \equiv 0(\text{mod} \quad 2)$. Then $y=1$ and $(x-p)^{'}|_{x=1+2y=3}=1\neq 0(\text{mod} \quad 2)$.

Let $x=1+2+4z$. Consider the solution of congruence equation $$3+4z-p\equiv 0(\text{mod} \quad 2^{3}).$$ Namely, $(3+4z)-(2^{m}t+8k-1)\equiv
0(\text{mod} \quad 2^{3})\quad (0\leq z<2)$, the solution of this
congruence equation is same to congruence equation $-2^{m-2}t-2k+1+z\equiv z+1 \equiv 0(\text{mod} \quad 2).$ Then $z=1$ and $(x-p)^{'}|_{x=3+4z=7}=1\neq 0(\text{mod} \quad 2)$.

Let $x=1+2+4+8w$. Consider the solution of congruence equation $$7+8w-p\equiv 0(\text{mod} \quad 2^{4}).$$ Namely, $(7+8w)-(2^{m}t+8k-1)\equiv
0(\text{mod} \quad 2^{4})\quad (0\leq w<2).$ The solution of this congruence equation is same to congruence equation $-2^{m-3}t-k+1+w\equiv -k+1+w \equiv 0(\text{mod} \quad 2)$.
If $k$ is even , then $w=1$; if $k$ is odd, then $w=0$. And $(x-p)^{'}|_{x=7+8w=(7\mbox{ or }15)}=1\neq 0(\text{mod} \quad 2)$.

When $p\equiv (8k-1)(\text{mod} \quad 2^{m}, m\geq 4)$, we can obtain 2-adic representation of $p$ modulo
$2^{m}$ is as following
$$
1+1\cdot 2+1\cdot 2^{2}+n_{1}\cdot 2^{3}+n_{2}\cdot
2^{4}+\cdots+n_{m-3}\cdot 2^{m-1},
$$
where $n_{1},n_{2},\cdots,n_{m-3}$ is 0 or 1.

\section{2-adic representation of $-p$ modulo $2^{m}$ (when $p\equiv (8k-1)(\text{mod} \quad 2^{m})\quad
(1\leq k \leq 2^{m-3}-1, m\geq 4)$)} \label{Z2mQR:Appendix:-p8k-1}

If we want to obtain the 2-adic representation of $-p$ modulo $2^{m}$,
we just need to find the 2-adic solution of equation $x+p=0$ modulo $2^{m}$.

Firstly, solve congruence equation $$x+p\equiv 0(\text{mod} \quad 2).$$
Namely, $x+(2^{m}t+8k-1)\equiv 0(\text{mod} \quad 2)\quad (0\leq x<2)$.
Then $x=1$ and $(x+p)^{'}|_{x=1}=1\neq 0(\text{mod} \quad 2)$.

Let $x=1+2y$. Consider the solution of congruence equation $$1+2y+p\equiv 0(\text{mod} \quad 2^{2}).$$ Namely, $(1+2y)+(2^{m}t+8k-1)\equiv
0(\text{mod} \quad 2^{2})\quad (0\leq y<2)$. The solution of this
congruence equation is same to congruence equation $2^{m-1}t+4k+y\equiv y \equiv 0(\text{mod} \quad 2)$. Then $y=0$ and
$(x+p)^{'}|_{x=1+2y=1}=1\neq 0(\text{mod} \quad 2)$.

Let $x=1+4z$. Consider the solution of congruence equation
$$1+4z+p\equiv 0(\text{mod} \quad 2^{3}).$$ Namely, $(1+4z)+(2^{m}t+8k-1)\equiv
0(\text{mod} \quad 2^{3})\quad (0\leq z<2)$. The solution of this
congruence equation is same to congruence equation
$2^{m-2}t+2k+z\equiv z \equiv 0(\text{mod} \quad 2)$. Then $z=0$ and
$(x+p)^{'}|_{x=1+4z=1}=1\neq 0(\text{mod} \quad 2)$.

Let $x=1+8w$. Consider the solution of congruence equation
$$1+8w+p\equiv 0(\text{mod} \quad 2^{4}).$$ Namely, $(1+8w)+(2^{m}t+8k-1)\equiv
0(\text{mod} \quad 2^{4})\quad (0\leq w<2)$. The solution of this
congruence equation is same to congruence equation
$2^{m-3}t+k+w\equiv k+w \equiv 0(\text{mod} \quad 2)$. If $k$ is
even then $w=0$; if $k$ is odd, then $w=1$. And
$(x+p)^{'}|_{x=1+8w=(1\mbox{or }9)}=1\neq 0(\text{mod} \quad 2)$.

When $p\equiv (8k-1)(\text{mod} \quad
2^{m}, m\geq 4)$, we can obtain the 2-adic representation of $-p$
modulo $2^{m}$ is as following
$$
1+n_{1}\cdot 2^{3}+n_{2}\cdot 2^{4}+\cdots+n_{m-3}\cdot 2^{m-1},
$$
where $n_{1},n_{2},\cdots,n_{m-3}$ is 0 or 1.

\section{2-adic representation of $\frac{1}{p}$ modulo $2^{m}$ (when $p\equiv (8k-1)(\text{mod} \quad 2^{m})\quad
(1\leq k \leq 2^{m-3}-1,m\geq 4)$)} \label{Z2mQR:Appendix:1/p8k-1}

If we want to obtain the 2-adic representation of $\frac{1}{p}$ modulo
$2^{m}$, we just need to find the 2-adic solution of equation $px-1=0$ modulo $2^{m}$.

Firstly, consider congruence equation $$px-1\equiv 0(\text{mod} \quad
2).$$ Namely, $(2^{m}t+8k-1)x-1\equiv 0(\text{mod} \quad 2)\quad (0\leq
x<2)$. Then $x=1$ and $(px-1)^{'}|_{x=1}=p\equiv 1\neq 0(\text{mod} \quad
2)\quad (\because p\equiv (8k-1)(\text{mod} \quad 2^{m})$.

Let $x=1+2y$. Consider congruence equation $$p(1+2y)-1\equiv
0(\text{mod} \quad 2^{2}).$$ Namely, $(2^{m}t+8k-1)(1+2y)-1\equiv 0(\text{mod}
\quad 2^{2})\quad (0\leq y<2).$ The solution of this congruence
equation is same to congruence equation $2^{m-1}t+4k-1+(2^{m}t+8k-1)y\equiv y+1 \equiv 0(\text{mod} \quad 2).$
Then $y=1$ and $(px-1)^{'}|_{x=1+2y=3}=p\neq 0(\text{mod} \quad 2)$.

Let $x=1+2+4z$. Consider congruence equation $$p(3+4z)-1\equiv
0(\text{mod} \quad 2^{3}).$$ Namely, $(2^{m}t+8k-1)(3+4z)-1\equiv 0(\text{mod}
\quad 2^{3})\quad (0\leq z<2).$ The solution of this congruence
equation is same to congruence equation $2^{m-2}t+2k-1+(2^{m}t+8k-1)z\equiv z+1 \equiv 0(\text{mod} \quad 2)$.
Then $z=1$ and $(px-1)^{'}|_{x=3+4z=7}=p\neq 0(\text{mod} \quad 2)$.

Let $x=1+2+4+8w$. Consider congruence equation $$p(7+8w)-1\equiv
0(\text{mod} \quad 2^{4}).$$ Namely, $(2^{m}t+8k-1)(7+8w)-1\equiv 0(\text{mod}
\quad 2^{4})\quad (0\leq w<2).$ The solution of this congruence
equation is same to congruence equation $2^{m-3}t+k-1+(2^{m}t+8k-1)w\equiv k-1+w \equiv 0(\text{mod} \quad 2)$.
If $k$ is even then $w=1$; if $k$ is odd then $w=0$. And
$(px-1)^{'}|_{x=7+8w=(7\mbox{or}15)}=p\neq 0(\text{mod} \quad
2)$.

When $p\equiv (8k-1)(\text{mod} \quad 2^{m})(m\geq 4)$, we obtain the 2-adic representation of
$\frac{1}{p}$ modulo $2^{m}$ is as following
$$
1+1\cdot 2+1\cdot 2^{2}+n_{1}\cdot 2^{3}+n_{2}\cdot
2^{4}+\cdots+n_{m-3}\cdot 2^{m-1},
$$
where $n_{1},n_{2},\cdots,n_{m-3}$ is 0 or 1го

\section{2-adic representation of $-\frac{1}{p}$ modulo $2^{m}$ (when $p\equiv (8k-1)(\text{mod} \quad 2^{m})\quad
(1\leq k \leq 2^{m-3}-1, m\geq 4)$)} \label{Z2mQR:Appendix:-1/p8k-1}

If we want to obtain the 2-adic representation of $-\frac{1}{p}$ modulo
$2^{m}$, we just need to find the 2-adic solution
of equation $px+1=0$ modulo $2^{m}$.

Firstly, consider the congruence equation $$px+1\equiv 0(\text{mod}
\quad 2).$$ Namely, $(2^{m}t+8k-1)x+1\equiv 0(\text{mod} \quad 2)\quad
(0\leq x<2).$ Then $x=1$ and $(px+1)^{'}|_{x=1}=p\equiv 1\neq
0(\text{mod} \quad 2)\quad (\because p\equiv (8k-1)(\text{mod} \quad
2^{m})$.

Let $x=1+2y$. Consider the congruence equation $$p(1+2y)+1\equiv
0(\text{mod} \quad 2^{2}).$$ Namely, $(2^{m}t+8k-1)(1+2y)+1\equiv 0(\text{mod}
\quad 2^{2})\quad (0\leq y<2).$ The solution of this congruence
equation is same to congruence equation
$2^{m-1}t+4k+(2^{m}t+8k-1)y\equiv y \equiv 0(\text{mod} \quad 2)$. Then
$y=0$ and $(px+1)^{'}|_{x=1+2y=1}=p\neq 0(\text{mod} \quad 2)$.

Let $x=1+4z$. Consider the congruence equation  $$p(1+4z)+1\equiv
0(\text{mod} \quad 2^{3}).$$ Namely, $(2^{m}t+8k-1)(1+4z)+1\equiv 0(\text{mod}
\quad 2^{3})\quad (0\leq z<2)$. The solution of this congruence
equation is same to congruence equation
$2^{m-2}t+2k+(2^{m}t+8k-1)z\equiv z \equiv 0(\text{mod} \quad 2)$. Then
$z=0$ and $(px+1)^{'}|_{x=1+4z=1}=p\neq 0(\text{mod} \quad 2)$.

Let $x=1+8w$. Consider the congruence equation
$$p(1+8w)+1\equiv 0(\text{mod} \quad 2^{4}).$$ Namely, $(2^{m}t+8k-1)(1+8w)+1\equiv 0(\text{mod} \quad 2^{4})\quad
(0\leq w<2)$. The solution of this congruence equation is same to
congruence equation $2^{m-3}t+k+(2^{m}t+8k-1)w\equiv k+w \equiv
0(\text{mod} \quad 2)$. If $k$ is even then $w=0$; if $k$ is odd
then $w=1$. And $(px+1)^{'}|_{x=1+8w=(1\mbox{or }9)}=p\neq 0(\text{mod}
\quad 2)$.

when $p\equiv
(8k-1)(\text{mod} \quad 2^{m})(m\geq 4)$, we can obtain the 2-adic representation of
$-\frac{1}{p}$ modulo $2^{m}$ is as followsing
$$
1+n_{1}\cdot 2^{3}+n_{2}\cdot 2^{4}+\cdots+n_{m-3}\cdot 2^{m-1},
$$
where $n_{1},n_{2},\cdots,n_{m-3}$ is 0 or 1.

% \appendix

% \section{}
% \label{}

\end{document}